\documentclass[11pt]{amsart}
\usepackage{amsmath,amssymb,hyperref}
\usepackage{amsfonts}
\usepackage{enumerate}
\usepackage{amsmath, amsthm}
\usepackage{amscd}
\input xy
\xyoption{all}
\hypersetup{pdfpagemode=FullScreen, colorlinks=true}

\newtheorem{thm}{Theorem}[section]

\newtheorem{prop}[thm]{Proposition}
\newtheorem{lemma}[thm]{Lemma}

\newtheorem{cor}[thm]{Corollary}
\newtheorem{defi}[thm]{Definition}

\newcommand{\br}{{\mathbb R}}

\newcommand{\ra}{\rightarrow}

\begin{document}
\title[Bounded cohomology]
{Bounded cohomology and negatively curved manifolds}

\author{Sungwoon Kim}
\address{School of Mathematics,
KIAS, Heogiro 85, Dongdaemun-gu, Seoul, 130-722, Republic of Korea}
\email{sungwoon@kias.re.kr}

\author{Inkang Kim}
\address{School of Mathematics
KIAS, Heogiro 85, Dongdaemun-gu Seoul, 130-722, Republic of Korea}
\email{inkang@kias.re.kr}

 \footnotetext[2]{The second author gratefully
acknowledges the partial support of  NRF grant
(R01-2008-000-10052-0).}
\date{}

\begin{abstract}

We study the bounded fundamental class in the top dimensional bounded cohomology of
negatively curved manifolds with infinite volume.
We prove that the bounded fundamental class of $M$ vanishes if $M$ is geometrically finite.
Furthermore, when $M$ is a $\mathbb{R}$-rank one locally symmetric space,
we show that the bounded fundamental class of $M$ vanishes if and only if
the Riemannian volume form on $M$ is the differential of a bounded differential form on $M$.
\end{abstract}

\maketitle

\section{Introduction}

The bounded cohomology $H^*_b(M,\mathbb{R})$ of a topological space $M$ is defined as
the cohomology of the subcomplex $C^*_b(M,\mathbb{R})$ of the singular cochain complex
$C^*(M,\mathbb{R})$ consisting of bounded cochains.
The inclusion of the complex $C^*_b(M,\mathbb{R})$ into $C^*(M,\mathbb{R})$
gives rise to the comparison map $$c^*_M : H^*_b(M,\mathbb{R})\rightarrow H^*(M,\mathbb{R}).$$

Bounded cohomology classes encode subtle properties of algebraic and
geometric nature of Riemannian manifolds. In particular, the second
bounded cohomology classes have been intensively studied. For
instance, Ghys \cite{Gh87} proved that the bounded Euler class in
$H^2_b(\Gamma,\mathbb{R})$, which is inherited by a $\Gamma$-action
on a circle, completely characterizes the action up to
semi-conjugacy. Burger, Iozzi, and Wienhard \cite{BIW07} showed
that, when $X$ is an irreducible Hermitian symmetric space not of
tube type, the Zariski-dense representations into the isometry group
of $X$ for any finitely generated group $\Gamma$ are classified up
to conjugacy by the bounded K\"{a}hler class. Furthermore, the
second bounded cohomology of a closed surface was studied by
Brooks-Series \cite{BS84}, Mitsumatsu \cite{Mi84}, Barge-Ghys
\cite{BG88}, and so on. The aim of the present paper is to study the
vanishing of the  bounded fundamental class in the top dimensional
bounded cohomology in relation to the geometric property of
negatively curved manifolds.

Let $M$ be an $n$-dimensional, connected, complete Riemannian manifold
with negative sectional curvature bounded away from zero.
Thurston \cite{Th78} introduced the geodesic straightening map
$Str_* :C_*(M,\mathbb{R})\rightarrow C_*(M,\mathbb{R})$ homotopic to the identity.
It is well-known that the volume of geodesic simplices in $M$ is uniformly bounded from above \cite{IY81}. Hence,
a bounded singular $n$-cocycle $\widehat{\omega}_M$ is defined by
$$\widehat{\omega}_M(\sigma)=\omega_M (Str_n(\sigma))=\int_{Str_n(\sigma)} \omega_M $$
for a singular simplex $\sigma :\Delta^n\rightarrow M$, where $\omega_M$ is the Riemannian volume form on $M$.
The bounded $n$-cocycle $\widehat{\omega}_M$ determines a bounded cohomology class $[\widehat{\omega}_M]$
in $H^n_b(M,\mathbb{R})$ representing the volume class in $H^n(M,\mathbb{R})$.
We call the bounded cohomology class $[\widehat{\omega}_M]\in H^n_b(M,\mathbb{R})$
\textit{the bounded fundamental class} of $M$.

One can easily notice that if $M$ is a closed manifold, the bounded
fundamental class of $M$ is non-trivial because the volume class of
$M$ is non-trivial in $H^n(M,\mathbb{R})$. However, the situation in
the case of open manifolds is a bit different. The volume class of
an open manifold $M$ vanishes in $H^n(M,\mathbb{R})$. In spite of
that, the bounded fundamental class may not vanish in the top
dimensional bounded cohomolgy of $M$. Hence, it can be expected that
the bounded fundamental class of $M$ encodes more specific geometric
properties of $M$.

A negatively curved Riemannian manifold $M$ is called {\it
geometrically finite} if a neighborhood of the quotient of the
convex hull of the limit set has a finite volume. See section
\ref{negative} for details. For a complete hyperbolic $3$-manifold
$M$ with infinite volume, Soma \cite{So97} showed that
$[\widehat{\omega}_M]=0$ in $H^3_b(M,\mathbb{R})$ if and only if $M$
is geometrically finite. We further explore the bounded fundamental
class of general negatively curved manifolds and first obtain a
sufficient condition for the bounded fundamental class to vanish as
follows.

\begin{thm}\label{thm1-1}
Let $M$ be an $n$-dimensional, negatively curved geometrically
finite manifold with infinite volume. Then, $[\widehat{\omega}_M]=0$
in $H_b^n(M,\mathbb{R})$.
\end{thm}

In particular, when $M$ is a three-dimensional, pinched negatively curved manifold with
infinite volume and positive injectivity radius, it turns out that
the bounded fundamental class of $M$
does not vanish in $H_b^3(M,\mathbb{R})$ if $M$ is not geometrically finite.
Thus, we can see that the bounded fundamental class of $M$ precisely encodes the geometrically
finiteness of $M$.

\begin{thm}\label{thm1-2}
Let $M$ be a complete, pinched negatively curved three-manifold with infinite volume
and positive injectivity radius.
Then, $[\widehat{\omega}_M]=0$ in $H^3_b(M,\mathbb{R})$
if and only if $M$ is geometrically finite.
\end{thm}

For a pinched negatively curved three-manifold $M$ with infinite
volume and positive injectivity radius, it can be shown that the
following conditions are equivalent.
\begin{itemize}
\item[(a)] $M$ is geometrically finite.
\item[(b)] The Cheeger isoperimetric constant of $M$ is strictly positive.
\item[(c)] The Riemannian volume form on $M$ admits a bounded primitive, i.e., the volume form is the differential of a bounded two form.
\end{itemize}
By the works of Bonahon \cite{Bo86} and Hou \cite{Ho03}, the conditions
$(a)$ and $(b)$ are equivalent.
Sikorav \cite{Sir} shows that the conditions $(b)$ and $(c)$ are equivalent
if $M$ has bounded geometry in
the sense that it is complete, its sectional curvature is bounded in absolute
value, and its injectivity radius is bounded below.
Recall that, given a differential form $\alpha$ on a Riemannian
manifold, if $d\beta =\alpha$ and  the norm of $\beta$ is
bounded, then $\beta$ is called a \textit{bounded primitive} of $\alpha$.
Therefore, it follows from Theorem \ref{thm1-2} that
$[\widehat{\omega}_M]=0$ in $H^3_b(M,\mathbb{R})$ is equivalent to
$(a)$, $(b)$, and $(c)$.

Although we get  four equivalent conditions only in the case of
pinched negatively curved three-manifolds with infinite volume and
positive injectivity radius, it gives us some clue to what
informations are encoded by the bounded fundamental class in
general. From Theorems \ref{thm1-1} and \ref{thm1-2}, one can expect
that the bounded fundamental class precisely encodes the positivity
of the Cheeger isoperimetric constant. Note that Theorem
\ref{thm1-1} supports the expectation because the Cheeger
isoperimetric constant of geometrically finite manifolds is strictly
positive by the work of Hamenst\"{a}dt \cite{Ha04}. For
$\mathbb{R}$-rank one locally symmetric spaces, we obtain a
necessary and sufficient condition for the bounded fundamental class
to vanish, which still supports the expectation.

\begin{thm}\label{thm1-3}
Let $M$ be an $n$-dimensional, $\mathbb{R}$-rank one locally symmetric space. Then,
$[\widehat{\omega}_M]=0$ in $H^n_b(M,\mathbb{R})$ if and only if the Riemannian volume form $\omega_M$ on $M$ is the differential of a bounded differential form on $M$.
\end{thm}

The question of when a differential form $\alpha$ admits a bounded
primitive has been raised by Gromov in \cite{Gr91} and \cite{Gr93}
when $\alpha$ is of degree $2$ and in \cite{Gr81} when $\alpha$ is
the volume form. In degree $2$, the existence of a bounded primitive
is strongly related to cofilling inequalities, hyperbolicity
phenomena. Gromov asserted that the existence of a bounded primitive
of the Riemannian volume form is closely related to the positivity
of the Cheeger isoperimetric constant \cite{Gr93}. In fact, Sullivan
\cite{Su76} asked whether openness at infinity of $M$ implies the
existence of a bounded primitive of its volume form. Recall that $M$
is called \textit{open at infinity} if there is a constant $C>0$
such that any compact domain $\Omega \subset M$ satisfies the
inequality $\text{Vol}(\Omega)\leq C \text{Vol}(\partial \Omega)$.
Note that a manifold with infinite volume is open at infinity if and
only if the Cheeger isoperimetric constant of the manifold is
strictly positive. Thus, a geometrically finite manifold with
infinite volume is open at infinity. From Theorems \ref{thm1-1} and
\ref{thm1-3}, we answer Sullivan's question for geometrically
finite, $\mathbb{R}$-rank one locally symmetric spaces in the
affirmative.

\begin{cor}\label{cor1-4}
Let $M$ be a geometrically finite, $\mathbb{R}$-rank one locally symmetric space
with infinite volume. Then, the Riemannian volume form on $M$ admits a bounded primitive.
\end{cor}

In the case of hyperbolic three-manifolds, we have the following
corollary from Theorem \ref{thm1-3}. This gives a proof of Gromov's
assertion in \cite{Gr93} for hyperbolic three-manifolds.

\begin{cor}\label{cor1-5}
Let $M$ be a complete hyperbolic three-manifold with infinite volume.
The following are equivalent.
\begin{itemize}
\item[(a)] $[\widehat{\omega}_M]=0$ in $H^3_b(M,\mathbb{R})$.
\item[(b)] The Cheeger isoperimetric constant of $M$ is strictly positive.
\item[(c)] The Riemannian volume form on $M$ admits a bounded primitive.
\item[(d)] $M$ is geometrically finite.
\end{itemize}
\end{cor}

\section{Bounded de Rham cohomology}

In this section, we review some results about the bounded de Rham cohomology of
symmetric spaces and then extend the results to simply-connected Riemannian
manifolds with negative sectional curvature bounded away from zero.

Let $\Omega^k(M)$ be the set of differential forms on a Riemannian manifold $M$.
For $\alpha \in \Omega^k(M)$, define the sup norm on $\Omega^k(M)$ as
$$\| \alpha \|_\infty = \sup_{x\in M} \sup_{v_1,\ldots,v_k \in T^1_xM}
|\alpha_x(v_1,\ldots,v_k)|,$$
where $T^1_x M$ is the unit tangent sphere of $M$ at $x\in M$. Consider the
space of bounded differential forms on $M$ whose differentials are also bounded,
$$\Omega^k_\infty (M)=\{ \alpha \in \Omega^k(M) \text{ }\big| \text{ }
\|\alpha \|_\infty < \infty \text{ and } \|d\alpha \|_\infty <
\infty \}.$$ By definition, the exterior differential satisfies
$d(\Omega^k_\infty(M)) \subset \Omega^{k+1}_\infty(M)$, so
$(\Omega^*_\infty(M),d)$ is a subcomplex of the standard de Rham
complex. The cohomology of this subcomplex is called \textit{bounded
de Rham cohomology}, denoted by $H^*_{bdR}(M,\mathbb{R})$.

Let $\mathcal{X}$ be a symmetric space of non-compact type of rank $r_\mathcal{X}$ and
$G$ be the identity component of the isometry group of $\mathcal{X}$.
Recently, Wienhard \cite{Wi11} showed the existence of a bounded primitive
of a closed bounded differential form on $\mathcal{X}$ with degree $k\geq r_\mathcal{X}+1$.
In particular, $G$-invariant differential forms on $\mathcal{X}$ admit
$P$-invariant bounded primitives for a minimal parabolic subgroup $P<G$.

\begin{prop}[Wienhard]
Let $\mathcal{X}$ be a symmetric space of non-compact type and
$G$ the connected component of the group of isometries. Let $\alpha$
be a $G$-invariant differential form on $\mathcal{X}$, then there exists a bounded
differential form $\beta$ with $d\beta =\alpha$. Moreover, $\beta$ can be
chosen to be $P$-invariant for a minimal parabolic subgroup $P<G$.
\end{prop}

Let $X$ be a complete, simply connected, Riemannian manifold
with negative sectional curvature bounded away from zero. Denote by $G$ the isometry group of $X$.
Let $d_X$ be the $G$-invariant Riemannian distance function on $X$.
The visual boundary $\partial X$ of $X$ is
defined as the set of asymptotic classes of geodesic rays. Let $\xi \in \partial X$ and
$\gamma : \mathbb{R} \rightarrow X$ a geodesic representing $\xi$. The Busemann function
$B_\xi : X\times X \rightarrow \mathbb{R}$ is defined by
$$B_\xi (x,y)=\lim_{t\rightarrow \infty}(d_X(x,\gamma (t))-d_X(y,\gamma (t))).$$

Let's set $B^0_\xi (x)=B_\xi (x,x_0)$ for a point $x_0\in X$.
The level sets of $B^0_\xi$ are horospheres centered at $\xi$.
Let $V_\xi = -\text{grad} B^0_\xi$ be the negative gradient vector field of $B^0_\xi$ and
$\phi^\xi_t :X\rightarrow X$ the corresponding flow. Then, $\|V_\xi \| =1$. Furthermore,
$V_\xi$ and $\phi^\xi_t$ are independent of the choice of the geodesic $\gamma$ representing
$\xi$.
By a similar argument in the proof of \cite{Wi11},
we can prove the following proposition.

\begin{prop}\label{pro2-2}
Let $X$ be a complete, simply connected, Riemannian manifold
with negative sectional curvature bounded away from zero and $\alpha$ a closed bounded smooth
differential form of degree $k$ on $X$. If $k\geq 2$, there exists
a bounded $(k-1)$-form $\beta$ on $X$ with $d\beta = \alpha$.
\end{prop}

\begin{proof}
We can assume that $X$ has all sectional curvatures at most $-1$ by scaling the metric.
Let $\xi$ be a point in $\partial X$. Let's define $\gamma_x :\mathbb{R} \rightarrow X$
for every $x\in X$ by $\gamma_x(t)=\phi^\xi_t(x)$. Note that $\gamma_x$ is the unique unit
speed geodesic passing through $x=\gamma_x(0)$, which is asymtotic to $\xi$ at infinity.
If a vector $v\in T_xX$ is orthogonal to $V_\xi$, then $(\phi^\xi_t)_*(v)$ is a Jacobi field
of geodesic variation by geodesic $\gamma_x$ that is perpendicular to $V_\xi$
and it decays uniformly exponentially. Thus, we have
\begin{eqnarray}\label{jac}
\| (\phi^\xi_t)_*(v) \| \leq e^{-t} \| v \|.
\end{eqnarray}

Let $v_1,\ldots,v_k \in T^1_x X$.
Each $v_i$ can be uniquely written by $v_i=v_i^P + v_i^T$,
where $v_i^P$ is a vector perpendicular to $V_\xi$ and $v_i^T$ is a
vector tangent to $V_\xi$.
Then, $(\phi^\xi_t)_*(v_i^P)$ is a Jacobi field along the geodesic
$\gamma_x$, which is perpendicular to $V_\xi$ for each $i=1,\ldots,k$, and so
it satisfies Equation (\ref{jac}).

Consider the set $S_i$ of all $i$-combinations of
a set $\{1,\ldots,k \}$ for each $i=1,\cdots,k$. Recall that a $i$-combination of the set $\{1,\ldots,k \}$ is
a subset of $i$ distinct elements of $\{1,\ldots,k \}$. Let's set
$$Q_i = \bigcup_{B\in S_i}\{ e=(e_1,\ldots,e_k) \text{ }|\text{ } e_j = v_j^P\text{ if }j\in
B\text{ and }e_j=v_j^T\text{ otherwise} \},$$ for each $i=1,\ldots,k$.
If $k\geq 2$, we have
{\setlength\arraycolsep{2pt}
\begin{eqnarray*}
|(\phi^\xi_t)^* \alpha(v_1,\ldots,v_k)| &=& |(\phi^\xi_t)^* \alpha
(v_1^P + v_1^T,\ldots,v_k^P + v_k^T)| \\
&=& \Big| \sum_{i=0}^k \sum_{e\in Q_i} \alpha((\phi^\xi_t)_* e) \Big| \\
&=& \Big| \sum_{e\in Q_{k-1}} \alpha((\phi^\xi_t)_* e) + \sum_{e\in Q_k}
\alpha((\phi^\xi_t)_* e) \Big| \\
&\leq & e^{-(k-1)t} \binom{k}{k-1} \|\alpha\|_\infty +e^{-kt} \binom{k}{k} \|\alpha\|_\infty \\
&= & e^{-(k-1)t} (k+e^{-t}) \| \alpha \|_\infty.
\end{eqnarray*}}
The third equation follows from $\alpha((\phi^\xi_t)_* e)=0$ for $e\in
Q_i$ if $i<k-1$. In fact, if $e\in Q_i$ for $i<k-1$, at least two
vectors in the coordinate of $(\phi^\xi_t)_* e$ are tangent to $V_\xi$. Thus,
$\alpha((\phi^\xi_t)_* e)=0$ because $\alpha$ is alternating.
From the above inequality, the integral
$$\alpha ' = -\int^\infty_0 (\phi^\xi_t)^* \alpha dt$$
is well-defined and, moreover, $\alpha '$ is a bounded differential form on $X$.
By the construction of $\alpha '$, we have $\mathfrak{L}_{V_\xi}\alpha '=\alpha$ and $d\alpha '
=0$,
where $\mathfrak{L}_{V_\xi}$ denotes the Lie derivative in the direction of $V_\xi$. We set
$ \beta = i_{V_\xi} \alpha ',$
where $i_{V_\xi}$ is the contraction of $\alpha '$ with the vector field $V_\xi$. Then
$$\alpha = \mathfrak{L}_{V_\xi} \alpha ' = d(i_{V_\xi}\alpha ')+ i_{V_\xi}(d \alpha ') = d
\beta.$$
Since $\alpha '$ is bounded and $V_\xi$ is a unit vector,
{\setlength\arraycolsep{2pt}
\begin{eqnarray*}
|\beta(v_1,\ldots,v_{k-1}) | &=& |i_{V_\xi}\alpha '(v_1,\ldots,v_{k-1}) | \\
&=& |\alpha '(V_\xi,v_1,\ldots,v_{k-1}) | \\
&\leq & \| \alpha ' \|_\infty \| V_\xi \| \|v_1\|\cdots \|v_{k-1} \| \\
&\leq & \| \alpha ' \|_\infty.
\end{eqnarray*}}
Therefore, we can conclude that $\beta$ is bounded.
\end{proof}

From Proposition \ref{pro2-2}, we can obtain an immediate corollary about
the bounded de Rham cohomology of $X$ as follows.

\begin{cor}\label{cor2-3}
Let $\alpha \in \Omega^k_\infty(X,\mathbb{R})$ be a closed bounded form for $k\geq 2$.
Then $[\alpha]=0$ in $H^k_{bdR}(X,\mathbb{R})$.
In particular, $H^k_{bdR}(X,\mathbb{R})=0$ for $k\geq 2$.
\end{cor}

Let $P$ be a parabolic subgroup of $G$, that is, the stabilizer group of
a point $\xi$ in $\partial X$.
In particular, $G$-invariant closed bounded forms on $X$ admit
$P$-invariant bounded primitives.

\begin{cor}\label{cor2-4}
Let $\alpha \in \Omega^k_\infty(X,\mathbb{R})^G$ be a $G$-invariant closed bounded
form for $k\geq 2$.
Then there exists a $P$-invariant differential form
$\beta \in \Omega^{k-1}_\infty (X,\mathbb{R})^P$ such that $d \beta =\alpha$
for a parabolic subgroup $P$ of $G$.
\end{cor}

\begin{proof}
Let $P=Stab_G(\xi)$ for a point $\xi$ in $\partial X$.
Since $\alpha$ is $G$-invariant, it is $P$-invariant. Furthermore, $V_\xi$ and $\phi^\xi_t$ are
$P$-invariant;
thus, the differential form
$$\beta=i_{V_\xi} \Big(-\int^\infty_0 (\phi^\xi_t)^*\alpha dt \Big)$$
constructed in the proof of Proposition \ref{pro2-2} is also a $P$-invariant differential form on
$X$.
\end{proof}

Note that the Riemannian volume form $\omega_X$ on $X$ is a $G$-invariant closed bounded form.
Thus, it admits a $P$-invariant bounded primitive for a parabolic subgroup $P<G$
by Corollary \ref{cor2-4}.

\section{Bounded fundamental class}

Let $M$ be an $n$-dimensional complete Riemannian manifold with
negative sectional curvature bounded away from zero.
Let $\omega_M$ be the Riemannian volume form on $M$. Recall that de Rham isomorphism $\Psi^* :
H^k_{dR}(M)\rightarrow H^k(M,\mathbb{R})$ is induced by
$$\Psi(\alpha)(\sigma)=\alpha (\sigma) =\int_\sigma \alpha,$$ where
$\alpha$ is a $k$-form on $M$ and $\sigma$ is a smooth $k$-simplex
in $M$. The volume class $[\omega_M]$ of $M$ in $H^n(M;\mathbb{R})$
is the cohomology class determined by the cocycle $\sigma\ra
\omega_M(\sigma)$. Furthermore, we have a cocycle
$\widehat{\omega}_M$ defined by
$$\widehat{\omega}_M (\sigma)=\omega_M (Str_n(\sigma))=\int_{Str_n(\sigma)}\omega_M$$
for any singular simplex $\sigma : \Delta^n \rightarrow M$, where
$Str_* :C_*(M,\mathbb{R})\rightarrow C_*(M,\mathbb{R})$ is the geodesic straightening map
introduced by Thurston \cite{Th78}.
Since the volume of geodesic simplices in $M$ is bounded from above \cite{IY81},
the cocycle $\widehat{\omega}_M$ is bounded, and so it determines a bounded
cohomology class in $H^n_b(M,\mathbb{R})$.

\begin{lemma} \label{lem4-1}
Let $M$ be a topological space and $C_*(M,\mathbb{R})$ be the singular chain complex of $M$.
If $\psi_* :C_*(M,\mathbb{R})\rightarrow C_*(M,\mathbb{R})$ is a chain map homotopic to the identity,
a cocycle $z:C_*(M,\mathbb{R})\rightarrow \mathbb{R}$ is cohomologous to $z\circ \psi_*$.
\end{lemma}

\begin{proof}
Let $H_* :C_*(M,\mathbb{R})\rightarrow C_{*+1}(M,\mathbb{R})$ be the chain homotopy between $\psi_*$ and
the identity on $C_*(M,\mathbb{R})$, that is,
$$\partial_{k+1}H_k + H_{k-1}\partial_k = \psi_k- id.$$
Given a $k$-cocycle $z:C_k(M,\mathbb{R})\rightarrow \mathbb{R}$, we have
{\setlength\arraycolsep{2pt}
\begin{eqnarray*}
z\circ \psi_k (\sigma) - z(\sigma) &=& z \circ (\psi_k -id) (\sigma) \\
&=& z \circ (\partial_{k+1}H_k + H_{k-1}\partial_k) (\sigma) \\
&=& z(\partial_{k+1}(H_k(\sigma)))+z\circ H_{k-1}(\partial_k \sigma) \\
&=& \delta (z\circ H_{k-1})(\sigma),
\end{eqnarray*}}
where $\delta$ is the coboundary operator on the singular cochain complex $C^*(M,\mathbb{R})$.
Thus, $z$ and $z\circ \psi_k$ represent the same cohomology class in $H^k(M,\mathbb{R})$.
\end{proof}

Since the geodesic straightening map is chain homotopic to the identity,
it follows from Lemma \ref{lem4-1} that $\omega_M
(\cdot)$ and $\widehat{\omega}_M = \omega_M \circ Str_n (\cdot)$
represent the same cohomology class in $H^n(M,\mathbb{R})$. In other words,
$[\widehat{\omega}_M]$ is a bounded cohomology class in $H^n_b(M,\mathbb{R})$
representing the volume class $[\omega_M]$ in $H^n(M,\mathbb{R})$.

\section{Negatively curved manifolds}\label{negative}

In this section, we study the bounded fundamental class of pinched negatively
curved manifolds. Let $X$ be an $n$-dimensional, simply-connected, complete Riemannian manifold
with pinched negative sectional curvature.
Suppose that $\Gamma$ is a discrete subgroup of the isometry group $G$ of $X$.
Denote the limit set of $\Gamma$ by $\Lambda_\Gamma$.
The convex hull $CH(\Lambda_\Gamma)$ of the limit set is $\Gamma$-invariant.
The convex core $\mathcal{C}_M$ of $M=\Gamma\backslash X$ is defined by
$\Gamma \backslash CH(\Lambda_\Gamma).$
If $\Gamma$ is not elementary, the convex core $\mathcal{C}_M$ is
the smallest convex submanifold of $M$ such that the inclusion map of $\mathcal{C}_M$ into $M$
is a homotopy equivalence.
Then, $M$ is called \textit{geometrically finite} if the volume
of the one-neighborhood $\mathcal{N}_1(\mathcal{C}_M)$ of
$\mathcal{C}_M$ in $M$ is finite. For a discussion of several equivalent definitions
of geometrically finiteness, we refer the reader to \cite{Bo95}.

\begin{thm}\label{thm5-1}
Let $M$ be an $n$-dimensional, pinched negatively curved
geometrically finite manifold with infinite volume. Then,
$[\widehat{\omega}_M]=0$ in $H_b^n(M,\mathbb{R})$.
\end{thm}

\begin{proof}
If $\Gamma$ is elementary, $\Gamma$ has an abelian subgroup of
finite index. Then, the top dimensional bounded cohomology of $M$ is
trivial \cite[Corollary 7.5.11]{Mo01}, and hence, $[\widehat{\omega}_M]=0$ in
$H^n_b(M,\mathbb{R})$.

Now let's assume that $\Gamma$ is not elementary.
Denote by $\omega_M$ the Riemannian volume form on $M$ and
by $\omega_X$ the Riemannian volume form on $X$.
In fact, the volume form $\omega_M$ on $M$ is
induced from $\omega_X$, that is, $\omega_M =\Gamma \backslash \omega_X$.
The top-dimensional de Rham cohomology of $M$ is trivial because $M$ is non-compact.
Thus, there exists a differential $(n-1)$-form $\zeta \in \Omega^{n-1}(M)$
such that $\omega_M = d \zeta$.

Let $\mathcal{P}_1,\ldots,\mathcal{P}_k$ be pairwise disjoint cusps of $M$.
Then, there exists a horoball $H_i$ in $X$ with center at infinity $\xi_i \in \partial X$
such that $\mathcal{P}_i$ is isometric to $\Gamma_i \backslash H_i$,
where $\Gamma_i$ is a subgroup of $\Gamma$ fixing $\xi_i$ for each $i=1,\ldots,k$.
Note that $\Gamma_i$ is contained in the parabolic subgroup
$P_i$ of $G$ fixing $\xi_i\in \partial X$.

By Corollary \ref{cor2-4},
there exists a bounded $P_i$-invariant differential $(n-1)$-form $\eta_i$ with $\omega_X=d \eta_i$,
and we have
$$ \omega_M |_{\mathcal{P}_i} = (\Gamma \backslash \omega_X)|_{\mathcal{P}_i}
= \Gamma_i \backslash (\omega_X |_{H_i})
=\Gamma_i \backslash (d \eta_i |_{H_i}) = d (\Gamma_i \backslash \eta_i) |_{\mathcal{P}_i},$$
for each $i=1,\ldots,k$. The last equation is possible because $\eta_i$ is $P_i$-invariant and,
therefore, $\Gamma_i$-invariant.
Let $\mathcal{P}=\mathcal{P}_1 \cup \cdots \cup \mathcal{P}_k$ be the disjoint union of $\mathcal{P}_i$'s and $\eta_\mathcal{P}$ be the differential
form on $\mathcal{P}$ with $\eta_{\mathcal{P}} |_{\mathcal{P}_i} = (\Gamma_i \backslash \eta_i)
|_{\mathcal{P}_i}$ for each $i=1,\ldots,k$.
Since $\eta_i$ is a bounded differential form on $\mathcal{P}_i$ for each $i=1,\cdots,k$,
$\eta_\mathcal{P}$ is also a bounded differential form on $\mathcal{P}$. Moreover,
we have $d\eta_\mathcal{P}=\omega_M |_\mathcal{P}$ and
{\setlength\arraycolsep{2pt} \begin{eqnarray*}
d ( \zeta |_{\mathcal{P}} - \eta_{\mathcal{P}} )&=& d \zeta |_{\mathcal{P}} - d \eta_{\mathcal{P}}
\\
&=& \omega_M |_{\mathcal{P}} - \omega_M |_{\mathcal{P}} =0.
\end{eqnarray*}}
Thus, $\zeta |_{\mathcal{P}} - \eta_{\mathcal{P}}$ is a closed $(n-1)$-form on $\mathcal{P}$.
From the following exact sequence,
$$\cdots \longrightarrow H^{n-1}_{dR}(M) \stackrel{i^*}{\longrightarrow}
H^{n-1}_{dR}(\mathcal{P}) \longrightarrow H^n_{dR}(M,\mathcal{P}),$$
there exists a closed $(n-1)$-form $\tau$ on $M$ such that
$$[\eta_{\mathcal{P}}- \zeta |_{\mathcal{P}}] = [i^*\tau].$$
Hence, there exists a $(n-2)$-form $\mu$ on $\mathcal{P}$ such that
$$ i^*\tau - (\eta_{\mathcal{P}}- \zeta |_{\mathcal{P}}) = d \mu.$$
One may extend $\mu$ to whole $M$; thus, we set $(n-1)$-form $\beta$ on $M$ by
$$\beta = \tau + \zeta -d\mu.$$
Then, we can check
$$d \beta = d \tau + d\zeta -d^2\mu = d \zeta = \omega_M,$$ and
{\setlength\arraycolsep{2pt}
\begin{eqnarray*}
\beta |_{\mathcal{P}} &=& \tau |_{\mathcal{P}} + \zeta |_{\mathcal{P}} -d\mu \\
&=& i^*\tau +\zeta |_{\mathcal{P}} -d\mu = \eta_{\mathcal{P}}.
\end{eqnarray*}}
Since $\eta_{\mathcal{P}}$ is bounded, $\beta |_{\mathcal{P}}$ is also bounded.
This means that there exists $L_1>0$ such that for every $x \in \mathcal{P}$,
$$\| \beta_x \| = \sup_{v_1,\ldots,v_{n-1} \in T^1_xM} |\beta_x(v_1,\ldots,v_{n-1})| \leq L_1.$$

Furthermore, since
$\mathcal{C}_M-\mathcal{P}$ is compact by the equivalent definition of the geometrically finite manifold,
there exists $L_2 >0$ such that the norm of $\beta$ over $\mathcal{C}_M-\mathcal{P}$
is bounded from above by $L_2$. Finally, we can conclude that
$$\| \beta_x \| < L,$$
for every $x\in \mathcal{C}_M$, where $L=\max \{L_1,L_2\}$.

For any $n$-simplex $\sigma$, we have
$\widehat{\omega}_M(\sigma)=\omega_M(Str_n(\sigma))=d\beta (Str_n(\sigma))
=\beta(\partial Str_n(\sigma))=\beta(Str_{n-1}(\partial \sigma))=\delta (\beta \circ Str_{n-1})(\sigma)$.
Since the homomorphism $i^* : H^n_b(M,\mathbb{R}) \rightarrow
H^n_b(\mathcal{C}_M,\mathbb{R})$ is an isomorphism, $[\widehat{\omega}_M]=0$ if and only if
$[i^*\widehat{\omega}_M]=0$, where $i:\mathcal{C}_M \rightarrow M$ is the inclusion.
Note that $\mathcal{C}_M$ is convex, and hence, $i^*\widehat{\omega}_M$ is well-defined.
Since the norm $\| \beta_x \|$ over $\mathcal{C}_M$ and the volume of geodesic simplices in $M$
are uniformly bounded from above, $i^*(\beta \circ Str_n) $ is
a bounded cochain of $C^{n-1}_b(\mathcal{C}_M,\mathbb{R})$,
which implies $[i^*\widehat{\omega}_M]=0$ in $H_b^n(\mathcal{C}_M,\mathbb{R})$.
Therefore, we can conclude $[\widehat{\omega}_M]=0$ in $H_b^n(M,\mathbb{R})$.
\end{proof}

In contrast to the case of geometrically finite manifolds, one knows
very little about Riemannian manifolds with infinite volume
that are not geometrically finite.
However, the situation in pinched negatively curved three-manifolds is completely different.

Let $M$ be a complete pinched negatively curved three-manifold with
a finitely generated fundamental group.
If $M$ has no cusps,
Agol \cite{Ag04} showed that $M$ is topologically tame; that is,
$M$ is homeomorphic to the interior of a compact manifold with boundary.

\begin{defi}
An end $E$ of $M$ is said to be \textit{geometrically infinite} if
there exists a divergent sequence of geodesics exiting $E$. An end
that is not geometrically infinite will be called
\textit{geometrically finite}.
\end{defi}

Bonahon \cite{Bo86} show that if an end $E$ is geometrically finite, then
there exists a neighborhood $U$ of $E$ which does not intersect the convex core of $M$
for the constant curvature case. However, it continues to hold for pinched negatively curved three-manifolds. It implies that if every end of $M$ is geometrically finite,
$M$ is geometrically finite. In other words, if $M$ is not geometrically finite,
$M$ must have a geometrically infinite end.

Let $U$ be a neighborhood of $E$.
If, for some surface $S_E$, $U$ is homeomorphic
to $S_E \times [0,\infty)$ and there exists a sequence of simplicial ruled surfaces
$(f_i : S_E \rightarrow U)_{i\in \mathbb{N}}$ such
that $f_i(S_E)$ is homotopic to $S_E \times 0$ in $U$ and leaves every compact subset of $M$,
then $E$ is said to be \textit{simply degenerate}.

\begin{thm}[Bonahon]
Let $M$ be a negatively pinched complete three-manifold with $\Gamma$ purely loxodromic.
If $M$ has boundary-irreducible compact core, then every geometrically infinite end of $M$ is simply degenerate.
\end{thm}

This result is also available without the assumption that $M$ has boundary-irreducible
compact core, which was proved by Hou \cite{Ho03}. Therefore, if $M$ is not geometrically finite,
$M$ has a simply degenerate end.

\begin{thm}\label{thm7-1}
Let $M$ be a complete, pinched negatively curved three-manifold with infinite volume
and positive injectivity radius.
Then, $[\widehat{\omega}_M]=0$ in $H^3_b(M,\mathbb{R})$
if and only if $M$ is geometrically finite.
\end{thm}

\begin{proof}
If $M$ is geometrically finite, then $[\widehat{\omega}_M]=0$ in
$H^3_b(M,\mathbb{R})$ by Theorem \ref{thm5-1}. Conversely, suppose
that $M$ is not geometrically finite and $[\widehat\omega_M]=0$ in
$H^3_b(M,\mathbb{R})$. We may assume that $M$ has sectional
curvature $\leq -1$ by scaling the metric. Due to the positive
injectivity radius of $M$, $M$ has no cusps. By the theorem of
Bonahon and Hou, $M$ must have a simply degenerate end, denoted by
$E$. From the definition of a simply degenerate end, there exists a
sequence of exiting simplicial ruled surfaces $(f_i:S\rightarrow
U)_{i \in \mathbb{N}}$. Then, one can assume that the images $\{
f_i(S) \}$ are disjoint and, moreover, $\mathcal{N}_1(f_i(S)) \cap
\mathcal{N}_1(f_j(S)) =\emptyset $ for all $i,j \in \mathbb{N}$ with
$i\neq j$, where $\mathcal{N}_r(A)$ is the $r$-neighborhood of $A$
in $M$ for a subset $A$ of $M$ and $r>0$.

For each simplicial ruled surface $f_i:S\rightarrow U$,
there exists a triangulation $T_i$ of $S$ such that each face of $T_i$ is
taken to a non-degenerate ruled
triangle and the total angle about each vertex is at least $2\pi$. Each ruled triangle has a
Riemannian metric of
curvature at most $-1$ and, thus, $f_i$ induces a piecewise Riemannian metric $\kappa_i$ on $S$
of curvature at most $-1$ with concentrated negative curvature at the vertices.
It is useful to consider $f_i:(S,\kappa_i)\rightarrow U$ to be a pathwise isometry.
It follows from the Gauss-Bonnet theorem for ruled triangles that we have
$$\text{Area}(S,\kappa_i) \leq 2 \pi |\chi(S)|.$$

Since $\text{inj}(S,\kappa_i) \geq \text{inj}(M) \geq r_0$ for some $r_0>0$, there exists a
constant $R>0$ depending only on $\chi (S)$ and $r_0$ such that
$$ \text{Diam}(S,\kappa_i) \leq R.$$

By the work of Freedman-Hass-Scott \cite{FHS83}, there exists an embedding $h_i:S\rightarrow U$
with image in $\mathcal{N}_1(f_i(S))$ that is homotopic to $f_i:S\rightarrow U$ for each $i\in \mathbb{N}$. In fact, $h_i$ is the minimal area surface with respect to a metric on $M$,
which is very large in the complement of $\mathcal{N}_1(f_i(S))$. Let $\tau_i$ be the metric
on $S$ induced by $h_i$. From the property of minimal area surface, we obtain
$$\text{Area}(S,\tau_i)\leq \text{Area}(S,\kappa_i) \leq 2 \pi |\chi(S)|.$$
We refer to \cite[Theorem 2.5]{CM96} for more details about this. In
addition, the diameter of $(S,\tau_i)$ is uniformly bounded from
above by $R$ because of $\text{inj}(S,\tau_i) \geq \text{inj}(M)
\geq r_0$. Note that $h_i(S) \cap h_j(S) = \emptyset \text{ for all
}i\neq j$ because $\mathcal{N}_1(f_i(S)) \cap \mathcal{N}_1(f_j(S))
=\emptyset $ for all $i,j \in \mathbb{N}$ with $i\neq j$.

Let $\mathcal{V}_i$ be a finite set of points in $(S,\tau_i)$ such that
\begin{enumerate}
\item any two points of $\mathcal{V}_i$ lie at distance at least $r_0/2$ from each other,
\item for all $x\in (S,\tau_i)$, there exists $p \in \mathcal{V}_i$ such that the distance from
    $x$ to $p$ is
less than $r_0/2$.
\end{enumerate}

A set $\mathcal{V}_i$ is obtained by successively marking points in $(S,\tau_i)$ at pairwise
distances $\geq r_0/2$
until there is no more room for such points. Due to Fejes T\'{o}th \cite{Fe53} and Buser
\cite{Bu92}, we can get
a triangulation $\mathcal{T}_i$ of $(S,\tau_i)$ as follows.

Consider circles with radius less than $r_0/2$ going through three
or more points of $\mathcal{V}_i$ but such that no point lies in the
interior of any of these circles. A continuity argument involving
families of circles with growing radii shows that any point of
$\mathcal{V}_i$ lies on such a circle. The radius of any circle in
the family is less than $r_0/2$. Moreover, if for each such circle
we draw the inscribed polygon whose vertices are the points of
$\mathcal{V}_i$ on that circle, then we get a tessellation of
$(S,\tau_i)$. We refer to \cite{Bu92} for more details on this
topic. Adding diagonals coming out of some fixed vertex on polygons
that are not triangles, we get a triangulation $\mathcal{T}_i$ of
$(S,\tau_i)$.

Now, we claim that the cardinality of $\mathcal{V}_i$ is uniformly bounded from above.
From the construction of $\mathcal{V}_i$, we have a family of two-dimensional small metric balls
$\{ B(x,r_0/4) \}_{x\in \mathcal{V}_i}$ that are embedded in $(S,\tau_i)$ and pairwise disjoint.
By the estimation of the area of a small metric ball in \cite{Cr09},
there exists a universal constant $a_0>0$ such that the area of small metric ball $B(x,r_0/4)$ can
be estimated by
$$ \text{Area}(B(x,r_0/4))\geq a_0 r_0^2.$$
Since the area of $(S,\tau_i)$ is uniformly bounded from above by $2 \pi |\chi(S)|$, we have
$$ 2\pi |\chi(S)| \geq \sum_{x\in \mathcal{V}_i} \text{Area}(B(x,\frac{r_0}{4})) \geq a_0 r_0^2
|\mathcal{V}_i|.$$
Thus, it gives an upper bound of the cardinality of $\mathcal{V}_i$ as follows:
$$|\mathcal{V}_i|\leq \frac{2\pi |\chi(S)|}{a_0 r_0^2 }.$$
By the uniform upper bound of $|\mathcal{V}_i|$, the number of faces of $\mathcal{T}_i$ is
also uniformly bounded from above by $N_0$ for some $N_0 \in \mathbb{N}$.

Let $M(i)$ be the region bounded by $h_1(S)$ and $h_i(S)$. By the Waldhausen's result \cite{Wa68},
$M(i)$ is homeomorphic to $S\times [0,1]$. Recall the Hauptvermutung conjecture,
that any two triangulations of a triangulable space have a common refinement, is
true for dimension $2$ (See \cite{Ra24}, \cite{Mo52}). Hence, we have a common refinement
$\mathcal{R}_i$ on $S$
of both $\mathcal{T}_1$ and $\mathcal{T}_i$ for each $i\in \mathbb{N}$.

Now, we will construct a triangulation $\mathcal{K}_i$ of $M(i)$
that restricts to triangulations corresponding to $\mathcal{T}_1$ and $\mathcal{T}_i$ of
the boundary of $M(i)$.
For convenience, we think of $M(i)$ as $S\times
[0,1]$.
First, cut $M(i)$ along $S\times \{1/2\}$.
Then, we get two pieces, $M_1(i)$ and $M_2(i)$, corresponding to $S\times [0,1/2]$ and
$S\times [1/2,1]$, respectively. Since $\mathcal{R}_i$ is a refinement of $\mathcal{T}_1$,
there exists a triangulation $\mathcal{K}_1(i)$ of $M_1(i)$ such that $\mathcal{K}_1(i)$
restricts to $\mathcal{T}_1$ of $S\times \{0 \}$
and restricts to $\mathcal{R}_i$ of $S\times \{ 1/2 \}$.
Similarly, there exists a triangulation $\mathcal{K}_2(i)$ of $M_2(i)$ such that $\mathcal{K}_2(i)$
restricts to
$\mathcal{R}_i$ of $S\times \{ 1/2 \}$
and restricts to $\mathcal{T}_i$ of $S\times \{1 \}$.
Since two triangulations $\mathcal{K}_1(i)$ and $\mathcal{K}_2(i)$ have the common triangulation on $S\times \{1/2\}$,
one can glue them and get a triangulation $\mathcal{K}_i$ of $M(i)$.
It follows from the construction of $\mathcal{K}_i$ that
$\mathcal{K}_i$ restricts to $\mathcal{T}_1$(resp. $\mathcal{T}_i$) of $S\times \{ 0\}$(resp. $S\times \{ 1\}$).

From the triangulation $\mathcal{K}_i$ of $M(i)$, we can obtain a
singular $3$-chain $z_i$ in $C_3(M,\mathbb{R})$ with $\partial z_i =
w_1 + w_i$, where $w_1$ (resp. $w_i$) is the $2$-chain in
$C_2(M,\mathbb{R})$ obtained from the triangulation $\mathcal{T}_1$
(resp. $\mathcal{T}_i$). Note that $w_1$ and $w_i$ are cycles in
$C_2(M,\mathbb{R})$ because $h_1(S)$ and $h_i(S)$ are homeomorphic
to the closed surface $S$.

Let
$H_*:C_*(M)\rightarrow C_{*+1}(M)$ be the chain homotopy from the geodesic straightening map
to the identity, that is, $\partial_{k+1} H_k +H_{k-1} \partial_k =Str_k -Id.$
Recall that the chain homotopy $H_k$ is defined by the straight line
homotopy between any $k$-simplex and its geodesically straightened
simplex.
Since $w_1$ and $w_i$ are cycles in $C_2(M,\mathbb{R})$, we have
{\setlength\arraycolsep{2pt}
\begin{eqnarray*}
\partial Str_3(z_i) &=& Str_2(\partial z_i)
= Str_2(w_1+w_i) \\
&=& \partial_3 H_2 (w_1+w_i) + H_1 \partial_2 (w_1+w_i) + (w_1+w_i) \\
&=& \partial_3 H_2 (w_1+w_i)+(w_1+w_i).
\end{eqnarray*}}
Any singular simplex $\sigma$ in $w_1$ or $w_i$ is contained in an
embedded ball $B$ in $M$ with  radius at most $r_0/2$. Since an
embedded ball in $M$ is geometrically convex, the image of
$Str_2(\sigma)$ is contained in the same ball $B$, and hence, the
image of the canonical straight line homotopy $H_2(\sigma)$ between
$\sigma$ and $Str_2(\sigma)$ is also totally contained in $B$. From
the definition of the canonical straight line homotopy, the volume
of $H_2(\sigma)$ cannot exceed the volume of the ball with radius
$r_0/2$; that is,
$$ |\omega_M (H_2(\sigma)) | = \Big|\int_{H_2(\sigma)} \omega_M \Big|\leq
B(\frac{r_0}{2}), $$
where $B(r_0/2)$ is the greatest volume of any ball of radius $r_0/2$ in $M$.
If $B_{\mathbb{H}^3}(r)$ denotes the volume of a ball of radius $r$
in the hyperbolic space $\mathbb{H}^3$, we have $B(r) \leq a_1 B_{\mathbb{H}^3}(r)$
for some constant $a_1>0$ depending only on sectional curvature of $M$.
Therefore, there exists a uniform constant $a_2>0$ such that
$$ |\omega_M (H_2(\sigma)) | \leq B(\frac{r_0}{2}) \leq a_1 B_{\mathbb{H}^3}(\frac{r_0}{2}) \leq
a_2.$$

The assumption $[\widehat{\omega}_M]=0$ in $H^3_b(M,\mathbb{R})$ implies that there exists a
$\gamma \in C^2_b(M,\mathbb{R})$ with $\widehat{\omega}_M =\delta \gamma$.
Since the number of simplices in both $w_1$ and $w_i$ is less than $N_0$, we have
\begin{eqnarray}\label{eqn1}
|\widehat{\omega}_M(z_i)|=|\delta \gamma (z_i)| = |\gamma (\partial z_i)| =|\gamma (w_1+w_i)|
\leq 2 \| \gamma \|_\infty N_0.
\end{eqnarray}

On the other hand, since $M$ is non-compact, there exists a $\beta \in \Omega^2(M)$ with $d\beta
=\omega_M$ and thus,
{\setlength\arraycolsep{2pt}
\begin{eqnarray*}
|\widehat{\omega}_M(z_i)|=\Big|\int_{Str_3(z_i)}\omega_M \Big| &=& \Big|\int_{Str_3(z_i)}d\beta
\Big| \\
&=& \Big|\int_{\partial Str_3(z_i)} \beta \Big| \\
&=& \Big|\int_{\partial_3 H_2(w_1+w_i)+(w_1+w_i)} \beta \Big| \\
&=& \Big|\int_{H_2(w_1+w_i)} \omega_M +\int_{M(i)} \omega_M \Big| \\
&\geq & \Big|\int_{M(i)} \omega_M \Big|-\Big|\int_{H_2(w_1+w_i)} \omega_M \Big| \\
&\geq & \text{Vol}(M(i))-2 N_0 a_2.
\end{eqnarray*}}
Since $\text{Vol}(M(i))$ goes to infinity as $i$ tends to infinity,
this contradicts inequality (\ref{eqn1}). This means that the
bounded fundamental class $[\widehat{\omega}_M]$ can never vanish if
$M$ is not geometrically finite, which completes the proof.
\end{proof}

\section{Continuous bounded cohomology and bounded differential forms}

In this section, we collect some definitions and results about the
continuous (bounded) cohomology of a group. We begin with the
definition of the continuous cohomology of a group.

\subsection{The continuous (bounded) cohomology}
Let $G$ be a topological group. The continuous cohomology of $G$ with coefficients in
$\mathbb{R}$,
denoted by $H^*_c(G,\mathbb{R})$, is defined as the cohomology of the complex
$$0 \longrightarrow C_c(G,\mathbb{R})^G \stackrel{\delta}{\longrightarrow}
C^2_c(G,\mathbb{R})^G \stackrel{\delta}{\longrightarrow}
C^3_c(G,\mathbb{R})^G \stackrel{\delta}{\longrightarrow} \cdots $$
with the usual homogeneous coboundary operator $\delta$, where
$$C^k_c(G,\mathbb{R})=\{f:G^{k+1} \rightarrow \mathbb{R}\text{ }|\text{ }f \text{ is continuous}
\}$$
and $C^k_c(G,\mathbb{R})^G$ denotes the subspace of $G$-invariant cochains.
This coboundary operator maps a $G$-invariant continuous function $f:G^{k+1}\rightarrow
\mathbb{R}$
to the $G$-invariant continuous function
$$\delta f(g_0,\ldots ,g_k)=\sum^k_{i=0} (-1)^i f(g_0,\ldots ,\hat{g_i},\ldots ,g_k).$$
For $f \in C^k_c(G,\mathbb{R})$, we define its sup norm as
$$ \| f \|_\infty = \sup \{ |f(g_0,\ldots,g_k)|\text{ }|\text{ }(g_0,\ldots,g_k) \in G^{k+1} \}.$$

The coboundary operator restricts to the complex $C_{c,b}^*(G,\mathbb{R})$ of continuous
bounded cochains with respect to the sup norm on $C^*_c(G,\mathbb{R})$ and
the continuous bounded cohomology $H^*_{c,b}(G,\mathbb{R})$ is defined as the
cohomology of this complex. The inclusion of complex $C_{c,b}^*(G,\mathbb{R}) \subset
C^*_c(G,\mathbb{R})$
induces a comparison map $$c^*_G : H^*_{c,b}(G,\mathbb{R})\rightarrow H^*_c (G,\mathbb{R}).$$
We refer the reader to \cite{Gui80} for continuous cohomology and \cite{Mo01} for
the bounded continuous cohomology.

Let $G$ be a locally compact group and $L$ be a closed subgroup of $G$. Then,
the cohomology
groups $H^*_c(L,\mathbb{R})$ and $H^*_{c,b}(L,\mathbb{R})$ are computed as the
cohomology of the complexes $C^*_c(G,\mathbb{R})^L$ and
$C_{c,b}^*(G,\mathbb{R})^L$, respectively. Furthermore, the sup norm
on $C^*_c(G,\mathbb{R})$ gives rise to the semi-norms in $H^*_c(L,\mathbb{R})$
and $H^*_{c,b}(L,\mathbb{R})$. The restriction maps
$res^* : H^*_c(G,\mathbb{R})\rightarrow H^*_c(L,\mathbb{R})
\text{ and } res^*_b : H^*_{c,b}(G,\mathbb{R})\rightarrow
H^*_{c,b}(L,\mathbb{R})$ induced by the inclusion $L < G$ are
realized at the cochain level by the inclusions $C^*_c(G,\mathbb{R})^G
\subset C^*_c(G,\mathbb{R})^L$ and $C_{c,b}^*(G,\mathbb{R})^G \subset
C_{c,b}^*(G,\mathbb{R})^L$. For more details, see \cite{Mo01}.

\subsection{Bounded differential forms}
Let $G$ be a connected semi-simple Lie group with finite center and associated symmetric space
$\mathcal{X}$.
Let $L$ be a closed subgroup of $G$. According to the Van Est isomorphism,
the continuous cohomology $H^*_c(L,\mathbb{R})$ is canonically isomorphic
to the cohomology $H^*(\Omega^* (\mathcal{X})^L)$ of the complex
$(\Omega^* (\mathcal{X})^L , d)$ of $L$-invariant differential forms on $\mathcal{X}$.
In particular, the continuous cohomology $H^*_c(G,\mathbb{R})$
is isomorphic to $\Omega^* (\mathcal{X})^G $ because
$G$-invariant differential forms on $\mathcal{X}$ are closed and bounded.
Moreover, $\Omega^* (\mathcal{X})^G $ in the top degree is one-dimensional generated by
the $G$-invariant volume form $\omega_\mathcal{X}$ on $\mathcal{X}$. Hence, we have the Van Est
isomorphism in top degree $n$:
$$\mathcal{J} : H^n_c(G,\mathbb{R}) \stackrel{\cong}{\longrightarrow}
\Omega ^n (\mathcal{X})^G = \mathbb{R} \cdot \omega_\mathcal{X} .$$

Unfortunately, there is no analogue of the Van Est isomorphism in the context of
continuous bounded cohomology. However, Burger and Iozzi \cite{BI07}
explore the relation between continuous bounded cohomology and
the complex of invariant bounded differential forms.

\begin{thm}[Burger and Iozzi]\label{BI}
Let $\Gamma$ be a torsion-free discrete subgroup of a connected semi-simple
Lie group $G$ with finite center and associated symmetric space $\mathcal{X}$.
Any class in the image of the comparison map
$$c^*_\Gamma : H^*_b(\Gamma,\mathbb{R})\rightarrow H^* (\Gamma,\mathbb{R})$$
is representable by a closed form on $\Gamma \backslash \mathcal{X}$ which is bounded.
\end{thm}

In fact, they construct a map $\delta^*_{\infty,L}:H^*_{c,b}(L,\mathbb{R})\rightarrow H^*(\Omega^*_\infty (\mathcal{X})^L)$ for any closed subgroup $L$ of $G$ such that the diagram
$$\xymatrix{
H^*_{c,b}(L,\mathbb{R}) \ar[r]^{c^*_L} \ar[rd]_{\delta^*_{\infty,L}} &
H^*_c(L,\mathbb{R}) &
H^*(\Omega^*(\mathcal{X})^L) \ar[l]_{\cong} \\
& H^*(\Omega^*_\infty(\mathcal{X})^L) \ar[ru]_{i^*_{\infty,L}} }$$
commutes, where $i^*_{\infty,L}$ is the map induced in cohomology by the inclusion of complexes
$$i^*_\infty : \Omega^*_\infty(\mathcal{X})\rightarrow \Omega^*(\mathcal{X}).$$
From this commutative diagram, it is seen that any class in the image of the comparison map
is representable by a bounded closed form.

\section{$\mathbb{R}$-rank one locally symmetric spaces}

In this section, we study the bounded fundamental class of
$\mathbb{R}$-rank one locally symmetric spaces $M$.
Let $\mathcal{X}$ be the symmetric universal
covering space of $M$ and $\Gamma$ be the fundamental group of $M$.
Denote by $G$ the identity component of the isometry group of $\mathcal{X}$.

\begin{thm}\label{thm6-1}
Let $M$ be an $n$-dimensional, $\mathbb{R}$-rank one locally symmetric space. Then,
$[\widehat{\omega}_M]=0$ in $H^n_b(M,\mathbb{R})$ if and only if the Riemannian volume form $\omega_M$ on $M$ is the differential of a bounded differential form on $M$.
\end{thm}

\begin{proof}
We choose convenient complexes for $H^*_{c,b}(G,\mathbb{R})$ and
$H^*_b(\Gamma,\mathbb{R})$.
From \cite[Corollary 7.4.10]{Mo01}, those continuous bounded cohomology groups
can be obtained as the cohomology of complexes $C^*_b(\mathcal{X},\mathbb{R})^G$ and
$C^*_b(\mathcal{X},\mathbb{R})^\Gamma$, respectively, with the canonical homogeneous
coboundary operator, where $C^*_b(\mathcal{X},\mathbb{R})$ is defined by
$$ C^k_b(\mathcal{X},\mathbb{R})=\{f:\mathcal{X}^{k+1}\rightarrow \mathbb{R}\text{ }|\text{ }
f \text{ is a continuous bounded function} \}.$$
The action of $G$ and $\Gamma$ are given by the natural diagonal action
on the Cartesian product $\mathcal{X}^{k+1}$.
The restriction map $res^*_b :H^*_{c,b}(G,\mathbb{R})\rightarrow H^*_b(\Gamma,\mathbb{R})$
induced by the inclusion $\Gamma <G$ is realized at the cochain level by the
inclusion $C^*_b(\mathcal{X},\mathbb{R})^G \subset C^*_b(\mathcal{X},\mathbb{R})^\Gamma$.

Let $\omega_\mathcal{X}$ be the $G$-invariant Riemannian volume form
on $\mathcal{X}$. Then, one can consider a cochain $\widehat{\omega}_\mathcal{X} :
\mathcal{X}^{n+1} \rightarrow \mathbb{R}$ in $C^*_b(\mathcal{X},\mathbb{R})$ defined by
$$\widehat{\omega}_\mathcal{X}(x_0,\ldots,x_n)=\int_{[x_0,\ldots,x_n]} \omega_\mathcal{X},$$
where $[x_0,\ldots,x_n]$ is the geodesic simplex with an ordered vertex set $\{ x_0,\ldots,x_n \}$ in
$\mathcal{X}$. Since the form $\omega_\mathcal{X}$ is a $G$-invariant closed bounded form and the volume of geodesic simplices in $\mathcal{X}$ is uniformly bounded from above, the cochain $\widehat{\omega}_\mathcal{X}$ is a $G$-invariant bounded cocycle.
Thus, it determines a bounded cohomology class $[\widehat{\omega}_\mathcal{X}]$ in $H^n_b(G,\mathbb{R})$.

In order to avoid confusion with $C^*_b(\mathcal{X},\mathbb{R})$ defined above,
let $B^*(\mathcal{X},\mathbb{R})$ be the subcomplex of the singular cochain complex of $\mathcal{X}$
consisting of bounded cochains. Then, the map $\pi^* : C^*_b(M,\mathbb{R}) \rightarrow B^*(\mathcal{X},\mathbb{R})$ establishes an isometric isomorphism between $C^*_b(M,\mathbb{R})$ and $B^*(\mathcal{X},\mathbb{R})^\Gamma$, and commutes with the differentials, where
$\pi : \mathcal{X} \rightarrow M$ is the universal covering of $M$.
From the definition of $\widehat{\omega}_M$,
$$\pi^* \widehat{\omega}_M (\tau) = \omega_M (Str_n \circ \pi(\tau))
= \omega_M (\pi \circ Str_n(\tau)) = \omega_\mathcal{X}(Str_n(\tau)),$$
for a singular simplex $\tau : \Delta^n \rightarrow \mathcal{X}$.
One can easily notice that $\pi^* \widehat{\omega}_M(\tau)$ depends only on the vertices of
the simplex $\tau : \Delta^n \rightarrow \mathcal{X}$.
Such a cochain is essentially a function $\mathcal{X}^{n+1} \rightarrow \mathbb{R}$, and so
one can consider the cocycle $\pi^* \widehat{\omega}_M$ as a function $\mathcal{X}^{n+1} \rightarrow \mathbb{R}$ defined by
$$ \pi^* \widehat{\omega}_M (x_0,\cdots,x_n) = \omega_\mathcal{X} ([x_0,\ldots,x_n])=\widehat{\omega}_\mathcal{X} (x_0,\ldots,x_n).$$
Thus, $\pi^*:H^*_b(M,\mathbb{R})\rightarrow
H^*(B^*(\mathcal{X},\mathbb{R})^\Gamma)\cong H^*_b(\Gamma,\mathbb{R})$ is an
isometric isomorphism with $\pi^*[\widehat{\omega}_M] = res^n_b([\widehat{\omega}_\mathcal{X}])$.

From \cite[Proposition 3.1]{BI07}, we have the following commutative diagram
for a discrete subgroup $\Gamma$ of $G$.
$$ \xymatrixcolsep{4pc}\xymatrix{
H^n_{c,b}(G,\mathbb{R}) \ar[r]^-{\delta^n_{\infty,G}} \ar[d]_-{res_b^n} &
\Omega_\infty^n(\mathcal X)^G=\mathbb{R}\cdot \omega_\mathcal{X} \ar[d]^-{res_\infty^n} \\
H^n_b(\Gamma,\mathbb{R}) \ar[r]^-{\delta^n_{\infty,\Gamma}} \ar[rd]_-{c^n_\Gamma} &
H^n(\Omega^*_\infty(\mathcal{X})^\Gamma) \ar[r]^-{i^n_{\infty,\Gamma}} &
H^n_{dR}(M,\mathbb{R}) \ar[ld]^-{\cong} \\
& H^*_c(\Gamma,\mathbb{R}) &
}$$


Note that $\Omega^n_\infty(\mathcal{X})^G=\Omega^n(\mathcal{X})^G=H^n_c(G,\mathbb{R})$ because
$G$-invariant differential forms on $\mathcal{X}$ are closed and bounded. Hence, one can easily
see $\delta^n_{\infty,G}=c^n_G$ by Theorem \ref{BI}.
Since $\delta^n_{\infty,G}=c^n_G$ and
$c^n_G([\widehat{\omega}_\mathcal{X}])=\omega_{\mathcal X}$,
$$\delta^n_{\infty,G}([\widehat{\omega}_\mathcal{X}])=\omega_{\mathcal
X}.$$

Now, suppose that $[\widehat{\omega}_M]=0$ in $H^n_b(M,\mathbb{R})$.
Then we have $$res^n_b([\widehat{\omega}_\mathcal{X}])=\pi^* [\widehat{\omega}_M]=0.$$
From the commutative diagram, we have
$$res_\infty^n(\omega_\mathcal{X})= res_\infty^n \circ \delta^n_{\infty,G}([\widehat{\omega}_\mathcal{X}])
= \delta^n_{\infty,\Gamma} \circ res_b^n
([\widehat{\omega}_\mathcal{X}])=0.$$ This means that there exists a
$\Gamma$-invariant bounded differential form $\beta$ on
$\mathcal{X}$ such that $\omega_\mathcal{X}=d\beta$. Since
$\omega_\mathcal{X}$ and $\beta$  descend to differential forms on
$M$ and $\omega_\mathcal{X}$ descends to the Riemannian volume form
$\omega_M$ on $M$, we can conclude that $\omega_M$ is the
differential of a bounded differential form on $M$.

Conversely, suppose that the volume form $\omega_M$ on $M$ is the differential of a bounded form
$\beta$ on $M$. Then
$\widehat{\omega}_M(\sigma)=\omega_M(Str_n(\sigma))=d\beta (Str_n(\sigma))
=\beta(\partial Str_n(\sigma))=\beta(Str_{n-1}(\partial \sigma))=\delta (\beta \circ Str_{n-1})(\sigma)$,
where $\delta :C^{n-1}(M,\mathbb{R})\rightarrow C^n(M,\mathbb{R})$ is the coboundary operator.
This shows $\widehat{\omega}_M= \delta(\beta \circ Str_{n-1})$. Since the volume of any geodesic
simplex in $M$ is
uniformly bounded from above, there exists $D>0$ such that
$$|\beta \circ Str_{n-1}(\sigma)| \leq \| \beta \|_\infty \cdot Vol(Str_{n-1}(\sigma)) \leq \| \beta
\|_\infty \cdot D.$$
This means that $\beta \circ Str_{n-1} \in C^{n-1}_b(M)$, and hence, $[\widehat{\omega}_M]=0$ in
$H^n_b(M,\mathbb{R})$.
\end{proof}

From Theorem \ref{thm6-1}, one can see that the bounded fundamental class is closely
related to the existence of a bounded primitive of the Riemannian volume form.
Also, the existence of a bounded primitive of the volume form is closely
related to the Cheeger isoperimetric constant. The Cheeger
isoperimetric constant $h(M)$ of a complete non-compact Riemannian
manifold $M$ is defined by
$$h(U)= \inf_{U\subset M} \frac{\text{Vol}(\partial U)}{\text{Vol}(U)},$$
where $U$ ranges over all connected, open submanifolds of $M$ with
compact closure and smooth boundary.
Gromov asserted that the positivity of the Cheeger isoperimetric constant of $M$ implies
the existence of a bounded primitive of the Riemannian volume form on $M$.
Thus, one can notice that the bounded fundamental class is closely related to
the Cheeger isoperimetric constant.
In fact, a geometrically finite manifold in Theorem \ref{thm5-1} is a kind of
manifold that has a positive Cheeger isoperimetric constant.

Let $M$ be a  negatively curved geometrically finite manifold of
infinite volume with at least three dimensions. Then, Hamenst\"{a}dt
\cite{Ha04} showed that the bottom of the $L^2$-spectrum of $M$ is
strictly positive. More precisely, there is a constant $c_0>0$
depending only on the dimension and curvature of $M$ such that the
bottom $\lambda_0(M)$ of the $L^2$-spectrum of $M$ satisfies
$$\lambda_0(M)\geq \frac{c_0}{\text{Vol}(\mathcal{N}_1(\mathcal{C}_M))^2}>0.$$

In addition, if the Ricci curvature of $M$ is uniformly bounded from
below, the Cheeger isoperimetric constant $h(M)$ is strictly
positive if and only if the bottom of the $L^2$-spectrum of $M$ is
positive \cite{Ji10}. Hence, the Cheeger isoperimetric constant of a
negatively curved geometrically finite manifold with infinite volume
is strictly positive.

It can be derived from Theorems \ref{thm5-1} and \ref{thm6-1} that
the Riemannian volume form on geometrically finite,
$\mathbb{R}$-rank one locally symmetric spaces admits a bounded primitive.

\begin{cor}\label{cor6-2}
Let $M$ be a geometrically finite, $\mathbb{R}$-rank one locally symmetric space
with infinite volume. Then, the Riemannian volume form on $M$ admits a bounded primitive.
\end{cor}

In particular, when $M$ is a complete hyperbolic three-manifold with infinite volume,
Canary \cite{Ca92} showed that
$M$ is geometrically finite if and only if the Cheeger isoperimetric constant of $M$ is positive.
Moreover, Soma \cite{So97} showed that $[\widehat{\omega}_M]=0$ in $H^3_b(M,\mathbb{R})$ if and
only if
$M$ is geometrically finite.
Applying our results to complete hyperbolic three-manifolds,
we can summarize four equivalent conditions as follows.

\begin{cor}\label{cor6-3}
Let $M$ be a complete hyperbolic three-manifold with infinite volume.
Then, the following are equivalent.
\begin{itemize}
\item[(a)] $[\widehat{\omega}_M]=0$ in $H^3_b(M,\mathbb{R})$.
\item[(b)] The Cheeger isoperimetric constant of $M$ is strictly positive.
\item[(c)] The Riemannian volume form on $M$ admits a bounded primitive.
\item[(d)] $M$ is geometrically finite.
\end{itemize}
\end{cor}

\begin{proof}
By the works of Soma and Canary, $(a)$, $(b)$, and $(d)$ are equivalent.
Theorem \ref{thm6-1} implies that $(a)$ is equivalent to $(c)$.
Therefore, the proof is completed. Furthermore, we can give an alternative
proof of Soma's result about the equivalence of $(a)$ and $(d)$ as follows.

If $M$ is elementary, the fundamental group $\Gamma$ of $M$
has an abelian subgroup of finite index.
Then, the bounded cohomology of $M$ in dimension $3$ is trivial, and hence, $[\widehat{\omega}_M]=0$
in $H^3_b(M,\mathbb{R})$.
If $M$ is non-elementary and geometrically finite,
$[\widehat{\omega}_M]=0$ in $H^3_b(M,\mathbb{R})$ by Theorem \ref{thm5-1}.

Conversely, suppose that $\Gamma$ is non-elementary and
$[\widehat{\omega}_M]=0$ in $H^3_b(M,\mathbb{R})$.
Then, the volume form on $M$ admits a bounded primitive by Theorem \ref{thm6-1}.
This implies that the Cheeger
isoperimetric constant $h(M)$ is strictly positive by Stokes' theorem.
By the work of Canary \cite{Ca92}, $M$ has to be geometrically finite.
\end{proof}

Among the four conditions in Corollary \ref{cor6-3}, we observe that $(b)$ does not imply $(d)$ in general.
Here is a counterexample. Let $M=\Gamma\backslash H^n_\mathbb{H}$ be a
quaternionic hyperbolic manifold of infinite volume which is not
geometrically finite. By \cite[Main theorem]{Le03}, the bottom of
the $L^2$-spectrum of $M$ is positive, which implies that the Cheeger
isoperimetric constant of $M$ is positive even though $M$ is not
geometrically finite. Such a geometrically infinite manifold can be
explicitly constructed as follows. Take a geometrically infinite
surface group hyperbolic three-manifold (a degenerate surface group).
Since $H^1_\mathbb{H}$ is isometric to $H^4_\br$, we can make the
surface group $\pi_1(S)\subset SO(3,1)\subset SO(4,1)$ stabilizes
$H^1_\mathbb{H}$, whose limit set is again contained in
$H^1_\mathbb{H}$. Then, it is easy to see that any
$\epsilon$-neighborhood of the convex core has an infinite volume.
Hence, we expect that the three conditions, $(a)$, $(b)$, and $(c)$, are equivalent
in general cases.

\end{document}